\documentclass{cimart}

\title{Moderate, large and super large deviations principles for Poisson process with uniform catastrophes}

\author{Artem Logachov, Olga Logachova and Anatoly Yambartsev    }

\authorinfo[A. Logachov]{Russian Presidential Academy of National Economy and Public Administration, Russia; Novosibirsk State Technical University, Russia}{logachev-av@ranepa.ru}

\authorinfo[O. Logachova]{Siberian State University of Geosystems and Technologies, Russia}{o.m.logacheva@sgugit.ru}

\authorinfo[A. Yambartsev]{University of São Paulo, Brazil}{yambar@usp.br}

\abstract{%
    In this paper, we expand and generalize the findings presented in our previous work on the law of large numbers and the large deviation principle for Poisson processes with uniform catastrophes. We study three distinct scalings: sublinear (moderate deviations), linear (large deviations), and superlinear (super large deviations). Across these scales, we establish different yet coherent rate functions.
    }


\VOLUME{33}
\YEAR{2025}
\ISSUE{1}
\NUMBER{8}
\DOI{https://doi.org/10.46298/cm.14900}

\begin{document}

% Here is where the main text should be typed:
\section{Introduction}

\noindent
\textbf{Stochastic processes with catastrophes.} In probability theory and mathematical modeling, population stochastic processes with catastrophes deal with an interplay between ``regular'' randomness in population change and abrupt (catastrophic) change when a significant part of the population disappears. Informally, following \cite{Stir1}, a random process with catastrophes can be represented as the difference between two components: a \textit{regular} component and a \textit{catastrophic} one. These processes serve as a tool for understanding and predicting various phenomena in diverse fields such as ecology and biology \cite{KA, K, RLS-TG, VAME}, economics, see \cite{Cav, Jord, Mart}, queueing systems \cite{Chao, DE1, JLL, LMC, ZSKA}, and more. There is much literature on processes involving catastrophes and we do not claim to provide a comprehensive list of relevant references.

A regular stochastic component represents the evolution of a system over time with minor changes in the value of the process. Frequently, these regular processes are Markov processes.

Markov processes with catastrophes received the attention of the probability community since the 1970's in the context of populational dynamic models. We address the reader to \cite{Brock} for historical review and motivation. The paper \cite{Brock} is probably the first systematic review and characterization of this class of processes. We also refer the reader to \cite{LYL1}, where we tried to provide an extended overview.

When the catastrophe time points form a Markov process, such as a Poisson process, the system retains its Markov property. This ensures that the future behavior of the process depends solely on its current state, simplifying both mathematical treatment and modeling.

\noindent
\textbf{Poisson processes with uniform catastrophes.} In this paper, we study the so-called Poisson processes with uniform catastrophes, as introduced in \cite{LYL}. The process consists of two main components: a Poisson process describing population growth and an independent Poisson point process representing the occurrence of catastrophes. At the time of a catastrophe, an eliminated portion is chosen uniformly, giving rise to the term uniform catastrophes. Since the sum of two independent Poisson processes is a Poisson process, we will need a single Poisson process to define the process, as described in the definition below; see (\ref{def}).

Let us give a formal definition of the process. In every direction below, we assume that all random elements are in the probability space $(\Omega,\mathfrak{F},\bf{P})$, and $\omega\in\Omega$ is an elementary event or a sample point.

Consider Markov chain $\eta(k)$, $k\in \mathbb{Z}^+$, $\mathbb{Z}^+=\{0\}\cup\mathbb{N}$, with the state space $\mathbb{Z}^+$ and transition probabilities
$$
{\bf P}(\eta(k+1)=j \mid \eta(k)=i)=
\left\{
\begin{array}{ll}
\frac{\lambda}{\lambda+\mu}, &\mbox{if }j=i+1,\\
\frac{\mu}{i(\lambda+\mu)}, &\mbox{if }0\leq j<i, i\neq 0,\\
1, &\mbox{if }j=1, i=0,
\end{array}
\right.
$$
where $\lambda$ and $\mu$ are positive constants, parameters of the model. We will always set $0$ as the initial state, $\eta(0)=0$.

Let $\nu(t)$, $t\in \mathbb{R}^+$ be a Poisson process with rate $\alpha$, and suppose that $\nu(\cdot)$ and $\eta(\cdot)$ are independent. The following random process $\xi(t)$ we will call the \textit{Poisson process with uniform catastrophes}:
\begin{equation}\label{defxi}
\xi(t):=\eta(\nu(t)), \ t\in \mathbb{R}^+.
\end{equation}
Suppose the current state of the process $\xi$ is $i$. At the next time point of the Poisson process $\nu(\cdot)$, the integer-valued process increases by one, $i \to i + 1$, with probability $\lambda / (\lambda + \mu)$, corresponding to the Poisson growth of the population (the regular component). Alternatively, with probability $\mu / (\lambda + \mu)$, the time point corresponds to a catastrophic event. In this case, the process uniformly selects the next state from $\{i - 1, i - 2, \ldots, 1, 0\}$, with probability $1/i$, eliminating a portion of the current population.

To study the asymptotic properties, we will consider the sequence of scaled processes
$$
\xi_T(t):=\frac{\xi(Tt)}{\varphi(T)}, \ t\in [0,1],
$$
where $T$ is a monotonically increasing unbounded parameter, i.e., $T \to \infty$, and $\varphi(T)$ is scaling function, $\varphi(T) \to \infty$ when $T\to\infty$.

\noindent
\textbf{Preliminary results.} In \cite{LYL} the law of large numbers (LLN) and the large deviation principle (LDP) on the phase state were proved for the case $\varphi(T)=T$. More precisely, consider the family of the processes $\xi_T(t):=\frac{\xi(Tt)}{T}, \ t\in [0,1]$. We proved that
\begin{enumerate}
\item[\bf LLN:]  for any $\varepsilon>0$ the following a.s. convergence takes place
	\begin{equation}\label{LLN1}	\mathbf{P}\Bigl(\lim\limits_{T\rightarrow\infty}\sup\limits_{t\in[0,1]}\xi_T(t)>\varepsilon\Bigr)=0,
	\end{equation}

\item[\bf LDP:]  the family of random variables $\xi_T(1)$ satisfies LDP, i.e. for any measurable closed set $C$ and open set $O$
$$
\limsup_{T \rightarrow \infty} \frac{1}{T} \ln \mathbf{P}(\xi_T(1) \in C ) \leq - I(C), \ \ \
\liminf_{T \rightarrow \infty} \frac{1}{T} \ln \mathbf{P}(\xi_T(1) \in O ) \geq -I(O),
$$
with the rate function
$$
	I(x)= \left\{
	\begin{array}{ll}
		\ \infty,  & \text{ if } x\in (-\infty,0),\\
		x\ln\left(\frac{\lambda+\mu}{\lambda}\right), & \text{ if } x\in [0,\alpha),\\
		x\ln\left(\frac{x(\lambda+\mu)}{\alpha\lambda}\right)-x+\alpha, & \text{ if } x\in [\alpha,\infty),\\
	\end{array}
	\right.
$$
where $I(B)=\inf_{x\in B}I(x)$, for any measurable set $B$.
\end{enumerate}

\noindent
\textbf{Three regimes.} Instead of the linear scaling $\frac{\xi(Tt)}{T}$ considered in \cite{LYL}, this paper examines a general scaling $\frac{\xi(Tt)}{\varphi(T)}$ under the following three regimes:

\begin{itemize} \item \textit{Sublinear} regime, when $\frac{\varphi(T)}{T} \to 0$, and under the additional condition (\ref{1.3}), which we will refer to as \textit{moderate} deviations. \item \textit{Linear} regime, where $\frac{\varphi(T)}{T} \to k > 0$, which we will refer to as \textit{large} deviations. \item \textit{Superlinear} regime, where $\frac{\varphi(T)}{T} \to \infty$, which we will refer to as \textit{superlarge} deviations. \end{itemize}

LDP holds in all three regimes, and we have determined the rate function for each case. Moreover, each space scaling function $\varphi(T)$ requires a corresponding scaling $\psi(T)$ for the logarithm of the probability in large deviations.

As expected, the rate function in the linear regime matches the one found in \cite{LYL}. In transition from the linear to the sublinear case, the rate function loses the $x\ln x$ term, and it becomes a linear function (for positive values) starting at the origin, with the slope depending on the parameters $\lambda$ and $\mu$. In the superlinear case, the rate function is a straight line (for positive values) and does not depend on the model parameters.

This paper completes the description of large deviations for Poisson processes with uniform catastrophes. The results of the paper are summarized in Table~\ref{tabela}.

\begin{table}
    \centering
    \begin{tabular}{l|c|l}
       $\varphi(T)$  & $\psi(T)$ & rate function \\
       \hline
sublinear & & \\
$\frac{\varphi(T)}{T} \to 0$ with (\ref{1.3}) & $\varphi(T)$ & $
I_1(x)= \left\{
\begin{array}{lcl}
	\ \infty,  \text{ if }\ x\in (-\infty,0),\\
	x\ln\left(\frac{\lambda+\mu}{\lambda}\right), \text{ if }\ x\in [0,\infty).\\
\end{array}
\right.
$
\\
 & & \\
\hline
linear & & \\
$\frac{\varphi(T)}{T} \to k>0$ & $\varphi(T)$ & $
J_k(x)= \left\{
\begin{array}{lcl}
	\ \infty,  \text{ if }\ x\in (-\infty,0),\\
	x\ln\left(\frac{\lambda+\mu}{\lambda}\right), \text{ if }\ x\in \big[0,\frac{\alpha}{k}\big),\\
	x\ln\left(\frac{kx(\lambda+\mu)}{\alpha\lambda}\right)-x+\frac{\alpha}{k}, \text{ if }\ x\in \big[\frac{\alpha}{k},\infty\big).\\
\end{array}
\right.
$ \\
 & & \\
\hline
superlinear & & \\
$\frac{\varphi(T)}{T} \to \infty$ & $\varphi(T)\ln\frac{\varphi(T)}{T}$ & $
I_2(x)= \left\{
\begin{array}{lcl}
	\ \infty,  \text{ if }\ x\in (-\infty,0),\\
	x, \text{ if }\ x\in [0,\infty).\\
\end{array}
\right.
$  \\
 & & \\
\hline    \end{tabular}
    \caption{Three regimes are defined by the asymptotic behavior of the space scale function $\varphi(T)$. For each regime, the normalized function $\psi(T)$, as a function of $\varphi(T)$, and the corresponding rate function were determined.}
    \label{tabela}
\end{table}

The paper is organized as follows. In the next section, we present the definitions and main results. Section~\ref{proof} proves the main theorems, and the final section contains some auxiliary results needed for the proofs of these theorems.

\section{Definitions and main results}\label{def}

In the proof of LDP, we will use the standard implication:
$$
\text{LLDP and ET} \Rightarrow \text{LDP},
$$
where LLDP is the local large deviation principle,
and ET is exponential tightness, see, for example, \cite[Lemma 4.1.23]{DZ}.

Recall the main definitions.
\begin{definition}
	A family of random variables $\xi_T(1)$ satisfies local large deviation principle (LLDP)  in  $\mathbb{R}$ with a rate function
	$I = I(x): \mathbb{R} \rightarrow [0,\infty]$ and the normalizing function
	$\psi(T)$, such that $\lim\limits_{T\rightarrow\infty}\psi(T) = \infty$,
	if the following equality holds for any  $x \in  \mathbb{R}$,
	$$
	\lim_{\varepsilon\rightarrow 0}\limsup_{T\rightarrow \infty}\frac{1}{\psi(T)}
	\ln\mathbf{P}(\xi_T(1)\in U_\varepsilon(x))
	=\lim_{\varepsilon\rightarrow 0}\liminf_{T\rightarrow \infty}\frac{1}{\psi(T)}
	\ln\mathbf{P}(\xi_T(1)\in U_\varepsilon(x))=-I(x),
	$$
	where
	$
	U_\varepsilon(x):=\{y\in \mathbb{R}: \ |x-y|<\varepsilon\}.
	$
\end{definition}

\begin{definition}
	A family of random variables $\xi_T(1)$ is
	exponentially tight  on  $\mathbb{R}$  if, for any $c < \infty$, there exists a compact set $K_c \subset \mathbb{R}$ such that
	$$
	\limsup_{T \rightarrow \infty} \frac{1}{\psi(T)} \ln \mathbf{P}(\xi_T(1) \not\in K_c  ) <-c.
	$$
\end{definition}

We denote the closure and interior of the set $B$ by  $[B]$ and $(B)$, respectively.

\begin{definition}
	A family of random variables $\xi_T(1)$ satisfies large deviation principle (LDP)  on  $\mathbb{R}$ with a rate function
	$I = I(x): \mathbb{R} \rightarrow [0,\infty]$ and normalizing function
	$\psi(T)$ such that $\lim_{T\rightarrow\infty}\psi(T) = \infty$,
	if for any $c \geq 0 $ \ the set $\{ x \in \mathbb{R}: I(x) \leq c \}$  is a compact set
	and, for any set $B \in \mathfrak{B}(\mathbb{R})$ the following inequalities hold:
	$$
	\limsup_{T \rightarrow \infty} \frac{1}{\psi(T)} \ln \mathbf{P}(\xi_T(1) \in B ) \leq - I([B]),
	$$
	$$
	\liminf_{T \rightarrow \infty} \frac{1}{\psi(T)} \ln \mathbf{P}(\xi_T(1) \in B )\geq -I((B)),
	$$
	where $\mathfrak{B}(\mathbb{R})$ is the Borel $\sigma$-algebra on $\mathbb{R}$, $I(B) = \inf\limits_{x \in B} I(x)$,
	$I(\emptyset) = \infty$.
\end{definition}

Further, we will use the following notations:
$\overline{B}$ is the complement of the set $B$;
$\mathbf{I}(B)$ is the indicator of the set $B$;
$\lfloor a \rfloor$ is integer part of the number $a$.

We study three regimes. The sublinear regime we define by the following two conditions:
\begin{equation} \label{1.2}
\lim\limits_{T\rightarrow \infty}\frac{\varphi(T)}{T}=0,
\end{equation}
and there exists $a\in (0,1)$ such that
\begin{equation} \label{1.3}
\ \ \ \lim\limits_{T\rightarrow \infty}\frac{\varphi(T)}{T^a}=\infty.
\end{equation}
The large deviation principle, in this case, we can refer to the moderate large deviations principle (MDP). The superlinear regime we define by the following condition,
\begin{equation} \label{1.4}
\lim\limits_{T\rightarrow \infty}\frac{\varphi(T)}{T}=\infty.
\end{equation}
The large deviation principle in this regime can refer to the super large deviation principle (SLDP). Finally, the linear regime can be associated with the classical large deviation principle. The liner regime we define by the following condition: there exists $k>0$, such that
\begin{equation} \label{15.11.2024-1}
\lim\limits_{T\rightarrow \infty}\frac{\varphi(T)}{T}=k.
\end{equation}

Now, we state the main theorems.

\begin{theorem}[\!\!{\cite[Theorem 1]{LRL}}] \label{t2.1} Let the condition (\ref{1.3}) holds, then for any $\varepsilon>0$
the following inequality holds
$$
\mathbf{P}\bigg(\lim\limits_{T\rightarrow\infty}\sup\limits_{t\in[0,1]}\xi_T(t)>\varepsilon\bigg)=0.
$$
\end{theorem}

Note that Theorem~\ref{t2.1} essentially states the law of large numbers. Specifically, it implies that under condition (\ref{1.3}), the measure of trajectories concentrates near 0 within the interval $[0,1]$.

\begin{theorem}[LDP] \label{t1.1}  Let the scaling $\varphi(T)$ satisfies the condition (\ref{15.11.2024-1}). %$\lim\limits_{T\rightarrow \infty}\frac{\varphi(T)}{T}=k>0$,
Then the family of random variables $\xi_T(1)$ satisfies
	LDP with normalization function
	$\psi(T)=\varphi(T)$ and with rate function
	\begin{equation}\label{22.11.5}
	J_k(x)= \left\{
	\begin{array}{lcl}
		\ \infty,  \text{ if }\ x\in (-\infty,0),\\
		x\ln\left(\frac{\lambda+\mu}{\lambda}\right), \text{ if }\ x\in \big[0,\frac{\alpha}{k}\big),\\
		x\ln\left(\frac{kx(\lambda+\mu)}{\alpha\lambda}\right)-x+\frac{\alpha}{k}, \text{ if }\ x\in \big[\frac{\alpha}{k},\infty\big).\\
	\end{array}
	\right.
	\end{equation}
\end{theorem}
\begin{theorem}[MDP] \label{t2.2}  Let the conditions (\ref{1.2}) and (\ref{1.3}) hold,
then the family of random variables $\xi_T(1)$ satisfies
large deviation principle with normalized function
$\psi(T)=\varphi(T)$ and the rate function
$$
I_1(x)= \left\{
           \begin{array}{lcl}
                              \ \infty,  \text{ if }\ x\in (-\infty,0),\\
                              x\ln\left(\frac{\lambda+\mu}{\lambda}\right), \text{ if }\ x\in [0,\infty).\\
                              \end{array}
                              \right.
$$
\end{theorem}

\begin{theorem}[SLDP] \label{t2.3}  Let the condition (\ref{1.4}) holds,
then the family of random variables $\xi_T(1)$ satisfies large deviation principle with normalized function
$\psi(T)=\varphi(T)\ln\frac{\varphi(T)}{T}$ and the rate function
$$
I_2(x)= \left\{
           \begin{array}{lcl}
                              \ \infty,  \text{ if }\ x\in (-\infty,0),\\
                              x, \text{ if }\ x\in [0,\infty).\\
                              \end{array}
                              \right.
$$
\end{theorem}

\section{Proof of main results}\label{proof}

\subsection{Proof of Theorem~\ref{t1.1}}

The proof from \cite{LYL} can be adapted by replacing $T$ with $\varphi(T)$. Therefore, we omit the proof here.

\subsection{Proof of Theorem~\ref{t2.2}}

According to the definition (\ref{defxi}), the Poisson process $\nu$ contains two parts: the regular growth of the population and the flow of catastrophic events. The process's classic representation, as a difference between regular and catastrophic components, will be more useful for the proof. Represent the random process
$\xi(t)$ in the following way (see \cite{LYL})
\begin{equation} \label{4.1}
\xi(t)=\nu_1(t)-\sum\limits_{k=0}^{\nu_2(t)}\zeta_k(\xi(\tau_k-)),
\end{equation}
where
\begin{itemize}
	\item $\nu_1(t)$, $\nu_2(t)$ are independent Poisson processes with rates $\frac{\alpha\lambda}{\lambda+\mu}$ and $\frac{\alpha\mu}{\lambda+\mu}$ correspondingly; $\nu_1$ corresponds to the regular increasing of a population, and $\nu_2$ corresponds to the flow of catastrophic events;
	\item $\tau_0=0$, and $\tau_1,\dots,\tau_k,\dots$ are the time jumps  of $\nu_2(t)$ (instances of catastrophic events);
	\item the random variables $\zeta_k(m)$, where $k \in \mathbb{Z}^+$ and $m \in \mathbb{Z}^+$, represent catastrophe sizes; these variables are mutually independent, independent of $\nu_1(t)$ and $\nu_2(t)$, and their distribution is uniform; specifically, for a population of size $m$, the catastrophe size follows
	$\mathbf{P}(\zeta_k(m)=r)=1/m$, $1\leq r\leq m$, $k\in \mathbb{N}$, $m\in \mathbb{N}$;
	\item
	we define $\zeta_k(0) = -1$ for $k \in \mathbb{N}$, which induces a reflection at $0$ occurring at rate $\alpha$; for convenience we define $\zeta_0(m)=0$, $m\in \mathbb{Z}^+$.
\end{itemize}

The main idea of the proof is to demonstrate that for any $x \geq 0$ and as $T \to \infty$, the probability
$$\mathbf{P}(\xi_T(1) \geq x)$$
is logarithmically equivalent to
$$\mathbf{P}\bigl( \nu_1(T) - \nu_1(T - c\varphi(T)) \geq x\varphi(T),\ \nu_2(T) - \nu_2(T - c\varphi(T)) = 0 \bigr),
$$
where $c = \frac{x}{\alpha}$. Roughly speaking, this means that if $\xi_T(1) \geq x$, then the trajectory of the process $\xi_T(t)$ must stay close to zero until time $t = 1 - \frac{c\varphi(T)}{T}$. Starting from this time point, at least $x\varphi(T)$ jumps must occur in the process $\nu_1(Tt)$, while no jumps (or $o(\varphi(T))$ jumps) happen in the process $\nu_2(Tt)$.

Unfortunately, the authors lack a simple proof of this fact, which is critical for establishing the upper bound (\ref{02.03.7}). The corresponding lower bound, on the other hand, can be derived in a much more straightforward manner (see formulas (\ref{new11.11.1}) and (\ref{23.11.24})).

\subsubsection*{Upper bound.}

Let us prove the upper bound,
\begin{equation}\label{02.03.7}
\limsup\limits_{T\rightarrow\infty}\frac{1}{\varphi(T)}\ln\mathbf{P}(\xi_T(1)\geq x)\leq
-I_1(x).
\end{equation}

Observe that Theorem~\ref{t2.1} establishes inequality (\ref{02.03.7}) for
$x=0$. Thus, it remains to prove this inequality for $x>0$. To this end, let us derive an upper bound for $\mathbf{P}(\xi_T(1)\geq x)$ when $x>0$. For any $\delta>0$, $c>0$, $n\in \mathbb{N}$ we have
\begin{equation}\label{23.11.18}
\begin{aligned}
&\mathbf{P}(\xi_T(1)\geq x) \\
&=\mathbf{P}\bigg(\xi_T(1)\geq x, \inf\limits_{t\in\big[1-\frac{c\varphi(T)}{T},1\big]}\xi_T(t)>\delta\bigg)
+\mathbf{P}\bigg(\xi_T(1)\geq x, \inf\limits_{t\in\big[1-\frac{c\varphi(T)}{T},1\big]}\xi_T(t)\leq\delta\bigg)
\\
&\leq
\mathbf{P}\bigg(\inf\limits_{t\in\big[1-\frac{c\varphi(T)}{T},1\big]}\xi_T(t)>\delta\bigg)
+\sum\limits_{r=1}^n\mathbf{P}\bigg(\xi_T(1)\geq x, \inf\limits_{t\in[t_{r-1},t_{r}]}\xi_T(t)\leq\delta\bigg)
=:\mathbf{P}_1+\mathbf{P}_2,
\end{aligned}
\end{equation}
where $t_r:=1-\frac{c\varphi(T)}{T}\left(1-\frac{r}{n}\right)$, $1\leq r \leq n$, $t_0:=1-\frac{c\varphi(T)}{T}$.

Let us establish an upper bound for $\mathbf{P}_1$. First, we divide the time interval $[0,1]$ into $u_T:=\lfloor \frac{T}{c\varphi(T)}\rfloor$ disjoint subintervals, defined by the increasing sequence of points $s_r$ for $0\le r \le u_T$, as follows:
$$
s_0:=0, \ s_r:=1-(u_T-r+1)\frac{c\varphi(T)}{T}, \  1\leq r \leq u_T,
\ s_{u_T+1}:=1.
$$
It is easy to see that
\begin{equation}\label{04.03.1}
\begin{aligned}
\left\{\inf\limits_{t\in\big[1-\frac{c\varphi(T)}{T},1\big]}\xi_T(t)>\delta\right\}
 & \subseteq \bigcup\limits_{r=0}^{u_T-1}
\left\{\inf\limits_{t\in[s_r,s_{r+1}]}
	\xi_T(t)\leq\delta,\inf\limits_{t\in [s_{r+1},1]}\xi_T(t)>\delta\right\}
\\
& \subseteq\bigcup\limits_{r=0}^{u_T-1}\left\{\inf\limits_{t\in[s_r,s_{r+1}]}
	\xi_T(t)\leq\delta,\inf\limits_{t\in[s_{r+1},s_{r+2}]}\xi_T(t)>\delta\right\}.
\end{aligned}
\end{equation}
The relations (\ref{04.03.1}) imply
\begin{equation}\label{A.1}
\begin{aligned}
	\mathbf{P}_1 &\leq \mathbf{P}\bigg(\bigcup\limits_{r=0}^{u_T-1}\left\{\inf\limits_{t\in[s_r,s_{r+1}]}
	\xi_T(t)\leq\delta,\inf\limits_{t\in[s_{r+1},s_{r+2}]}\xi_T(t)>\delta\right\}\bigg) \\
	& \leq\sum\limits_{r=0}^{u_T-1}\mathbf{P}\bigg(\inf\limits_{t\in[s_r,s_{r+1}]}
	\xi_T(t)\leq\delta,\inf\limits_{t\in[s_{r+1},s_{r+2}]}\xi_T(t)>\delta\bigg).
\end{aligned}
\end{equation}
For any $r=0, 1, \ldots, u_T-1$, let $A_r$ denote the event
$$
A_r:=\bigl\{\nu_1(Ts_{r+2})-\nu_1(Ts_{r})<2a\varphi(T)\bigr\},
$$
where $a>0$ is a constant to be specified below.
Using Markov's inequality, we obtain an upper bound for the probability of $\overline{A_r}$ as follows:
$$
\mathbf{P}\bigl( \nu_1(Ts_{r+2})-\nu_1(Ts_r)\geq 2a\varphi(T) \bigr) \le
\frac{\mathbf{E}e^{\nu_1(T(s_{r+2}-s_r))}}{e^{2a\varphi(T)}}\leq
e^{\frac{3c\alpha\lambda}{\lambda+\mu}(e-1)\varphi(T)-2a\varphi(T)}.
$$
Choosing $a>\frac{3c\alpha\lambda}{\lambda+\mu}(e-1)$ we obtain
$$
\mathbf{P}\bigl( \nu_1(Ts_{r+2})-\nu_1(Ts_r)\geq 2a\varphi(T) \bigr) \le
e^{-a\varphi(T)}.
$$
Returning to (\ref{A.1}), we have
\begin{equation}\label{A.2}
\begin{aligned}
	\mathbf{P}_1 & \leq \sum\limits_{r=0}^{u_T-1} \bigg( \mathbf{P} \bigg( \inf\limits_{t\in[s_r,s_{r+1}]}\xi_T(t)\leq\delta,
	\inf\limits_{t\in[s_{r+1},s_{r+2}]}\xi_T(t)>\delta,
	A_r\bigg) + \mathbf{P} \left( \overline{A_r}\right) \bigg) \\
	 & \leq \sum\limits_{r=0}^{u_T-1} \mathbf{P} \bigg( \inf\limits_{t\in[s_r,s_{r+1}]}\xi_T(t)\leq\delta,
\inf\limits_{t\in[s_{r+1},s_{r+2}]}\xi_T(t)>\delta,
A_r\bigg) + u_Te^{-a\varphi(T)} \\
	& \le u_T \max\limits_{0\leq r \leq u_T-1}\mathbf{P}\bigg(\inf\limits_{t\in[s_r,s_{r+1}]}\xi_T(t)\leq\delta,
	\inf\limits_{t\in[s_{r+1},s_{r+2}]}\xi_T(t)>\delta,
	A_r\bigg) + u_Te^{-a\varphi(T)} \\
	& =:u_T\max\limits_{0\leq r \leq u_T-1}\mathbf{P}_{1r}+ u_Te^{-a\varphi(T)}. \phantom{\bigg(}
\end{aligned}
\end{equation}

To estimate $\mathbf{P}_{1r}$, the number of catastrophic events and their associated disaster sizes should be considered in the calculations. For any $r=0, 1, \ldots, u_T-1$, let
$$
\begin{aligned}
B_r:= &\bigl\{\nu_2(Ts_{r+2})-\nu_2(Ts_{r+1})< v\varphi(T)\bigr\}, \\
C_r:= &\left\{	
\sum\limits_{k=0}^{\nu_2(T)}\zeta_k(\xi(\tau_k-))\mathbf{I}(\tau_k\in[s_{r+1},s_{r+2}]) < 2a\varphi(T) \right\}.
\end{aligned}
$$
Using Lemma~\ref{l5.4} with $u\in[0,1)$ we obtain the inequality
$$
\mathbf{P}(B_r)\leq\mathbf{P}\bigl(\nu_2\left(c\varphi(T)\right)< v\varphi(T)\bigr)\leq
\exp\left\{-\varphi(T)\left(\frac{\alpha\mu c}{\lambda+\mu}(1-u)+v\ln u\right)\right\}.
$$
By choosing $u=\frac{1}{3}$ and $v=\frac{\alpha\mu c}{6(\lambda+\mu)\ln 3}$, we obtain the following upper bound:
\begin{equation} \label{23.11.11}
\mathbf{P}(B_r)\leq e^{-\frac{\alpha\mu c}{2(\lambda+\mu)}\varphi(T)}.
\end{equation}
Observe that, due to (\ref{4.1}), it is impossible for $\xi_T$ to have an infimum less than $\delta$ in the interval $[s_{r}, s_{r+1}]$, and simultaneously to have at most $2a\varphi(T)$ upward jumps in the interval $[s_r, s_{r+2}]$ (the event $A_r$), while also having a total disaster size of at least $2a\varphi(T)$ in the interval $[s_{r+1}, s_{r+2}]$ (the event $\overline{C_r}$), and subsequently remaining greater than $\delta$ during the time interval $[s_{r+1}, s_{r+2}]$:
$$
\mathbf{P}\Bigl(\ \inf\limits_{t\in[s_r,s_{r+1}]}\xi_T(t)\leq\delta,\inf\limits_{t\in[s_{r+1},s_{r+2}]}\xi_T(t)>\delta,A_r, \overline{C_r}\Bigr)=0.
$$
Thus, using (\ref{23.11.11}), we obtain
\begin{equation} \label{new3.2}
\begin{aligned}
\mathbf{P}_{1r} & = \mathbf{P}\Bigl(\ \inf\limits_{t\in[s_r,s_{r+1}]}\xi_T(t)\leq\delta,\inf\limits_{t\in[s_{r+1},s_{r+2}]}\xi_T(t)>\delta,A_r, C_r\Bigr) \\
&\leq  \mathbf{P}\Bigl(\ \inf\limits_{t\in[s_{r+1},s_{r+2}]}\xi_T(t)>\delta,A_r,C_r\Bigr) \\
&\leq\mathbf{P}\Bigl(\ \inf\limits_{t\in[s_{r+1},s_{r+2}]}\xi_T(t)>\delta,
A_r, \overline{B_r},C_r\Bigr)
+\mathbf{P}(B_r) \\
&\leq \mathbf{P}\Bigl(\ \inf\limits_{t\in[s_{r+1},s_{r+2}]}\xi_T(t)>\delta,A_r, \overline{B_r},C_r\Bigr) +e^{-\frac{\alpha\mu c}{2(\lambda+\mu)}\varphi(T)} =:\mathbf{P}_{11r}+e^{-\frac{\alpha\mu c}{2(\lambda+\mu)}\varphi(T)}. \phantom{\bigg(}
\end{aligned}
\end{equation}

To find an upper bound for $\mathbf{P}_{11r}$, we derive an exponential bound by considering the combination of a large number of catastrophic events, $\overline{B_r}$, and an upper bound for the total catastrophic size, $C_r$. The event $A_r$ will be disregarded in this estimation.

To begin, let us first truncate the sizes of the catastrophes. Define the random variables
$$
\check{\zeta}_{k}(m):=\left\{
           \begin{array}{rcl}
                              \zeta_{k}(m), &\text{if}& \zeta_{k}(m)\leq \lfloor\delta \varphi(T)\rfloor,\\
                              \gamma_{k}, &\text{if}& \zeta_{k}(m)> \lfloor\delta \varphi(T)\rfloor,\\
                              \end{array}
                              \right.
$$
where the random variables $\gamma_{k}$, $k\in \mathbb{N}$, are mutually independent, follow a uniform distribution with $\mathbf{P}(\gamma_{k}=r)=\frac{r}{\lfloor\delta \varphi(T)\rfloor}$, for $1\leq r \leq \lfloor\delta \varphi(T)\rfloor$, and do not depend on $\zeta_{k}(m)$, $k\in \mathbb{N}$, or on $\nu_1(\cdot)$ and $\nu_2(\cdot)$. Note that for any $m \ge \lfloor\delta \varphi(T)\rfloor$
\begin{equation}\label{02.03.8}
\check{\zeta}_{k}(m) = \zeta_k(\lfloor \delta\varphi(T)\rfloor) \mbox{ in distribution}.
\end{equation}
Denote by $\kappa_1,\dots,\kappa_{\lfloor v\varphi(T)\rfloor}$
the first $\lfloor v\varphi(T)\rfloor$ jump instances of the process $\nu_2(Tt)$ during the time interval $[s_{r+1},\infty)$
(these time moments lie within the interval $[s_{r+1},s_{r+2}]$ under the event $\overline{B_r}$), and define $V$ to be the set
$$
\left\{h_1\in \mathbb{Z}^+,\dots,h_{\lfloor v\varphi(T)\rfloor}\in \mathbb{Z}^+:h_1+\dots+h_{\lfloor v\varphi(T)\rfloor}\leq 2a\varphi(T),
\max_{1 \leq l \leq \lfloor v\varphi(T)\rfloor}h_l\leq  \lfloor\delta \varphi(T)\rfloor\right\}.
$$
Using the notation introduced above, we obtain 
$$\mathbf{P}_{11r}  \leq
\mathbf{P}\bigg(\xi(\kappa_1-)>\delta\varphi(T),\dots,\xi(\kappa_{\lfloor v\varphi(T)\rfloor}-)>\delta\varphi(T),
\sum\limits_{k=1}^{\lfloor v\varphi(T)\rfloor}\zeta_k(\xi(\kappa_k-))\leq 2a\varphi(T), \overline{B_r}\bigg)$$
$$ \leq \mathbf{P}\bigg(\xi(\kappa_1-)>\delta\varphi(T),\dots,\xi(\kappa_{\lfloor v\varphi(T)\rfloor}-)>\delta\varphi(T),
\sum\limits_{k=1}^{\lfloor v\varphi(T)\rfloor}\check{\zeta}_k(\xi(\kappa_k-))\leq 2a\varphi(T)\bigg)$$
\begin{equation}\label{02.03.1}
\begin{aligned}
& =\sum\limits_{\substack{
\lfloor\delta \varphi(T)\rfloor<m_{1}<\infty\\
\dots\\
\lfloor\delta \varphi(T)\rfloor<m_{\lfloor v\varphi(T)\rfloor}<\infty
}}
\mathbf{P}\left(\bigcap\limits_{k=1}^{\lfloor v\varphi(T)\rfloor}\{\xi(\kappa_k-)=m_k\},
\sum\limits_{k=1}^{\lfloor v\varphi(T)\rfloor}\check{\zeta}_{k_k}(m_{k})\leq 2a\varphi(T)\right)
\\
& =\sum\limits_{(h_1,\dots,h_{\lfloor v\varphi(T)\rfloor})\in V}\sum\limits_{\substack{
\lfloor\delta \varphi(T)\rfloor<m_{1}<\infty\\
\dots\\
\lfloor\delta \varphi(T)\rfloor<m_{\lfloor v\varphi(T)\rfloor}<\infty
}}
\mathbf{P}\left(\bigcap\limits_{k=1}^{\lfloor v\varphi(T)\rfloor}\{\xi(\kappa_k-)=m_k\},
\bigcap\limits_{k=1}^{\lfloor v\varphi(T)\rfloor}\{\check{\zeta}_{k_k}(m_{k})= h_k\}\right)
\\
& =:\sum\limits_{(h_1,\dots,h_{\lfloor v\varphi(T)\rfloor})\in V}\sum\limits_{\substack{
\lfloor\delta \varphi(T)\rfloor<m_{1}<\infty\\
\dots\\
\lfloor\delta \varphi(T)\rfloor<m_{\lfloor v\varphi(T)\rfloor}<\infty
}}\mathbf{P}_{\lfloor cT\rfloor}(h_1,\dots,h_{\lfloor v\varphi(T)\rfloor}).
\end{aligned}
\end{equation}
Let us bound the $\mathbf{P}_{\lfloor cT\rfloor}(h_1,\dots,h_{\lfloor v\varphi(T)\rfloor})$.
Using the fact that $\check{\zeta}_{k_{\lfloor v\varphi(T)\rfloor}}(m_{\lfloor v\varphi(T)\rfloor})$ is independent of the events
$\bigcap\limits_{k=1}^{\lfloor v\varphi(T)\rfloor}\{\xi(\kappa_k-)=m_k\}$ and $
\bigcap\limits_{k=1}^{\lfloor v\varphi(T)\rfloor-1}\{\check{\zeta}_{k_k}(m_{k})= h_k\}$, we obtain
\begin{equation}\label{02.03.2}
\begin{aligned}
\mathbf{P}_{\lfloor cT\rfloor}&(h_1,\dots,h_{\lfloor v\varphi(T)\rfloor}) \\
=&
\mathbf{P}\left(\check{\zeta}_{k_{\lfloor v\varphi(T)\rfloor}}(m_{\lfloor v\varphi(T)\rfloor}))=h_{\lfloor v\varphi(T)\rfloor}\right)
\\ & \times
\mathbf{P}\left(\bigcap\limits_{k=1}^{\lfloor v\varphi(T)\rfloor}\{\xi(\kappa_k-)=m_k\},
\bigcap\limits_{k=1}^{\lfloor v\varphi(T)\rfloor-1}\{\check{\zeta}_{k_k}(m_{k})= h_k\}\right)
\\
=& \frac{1}{\lfloor \delta\varphi(T)\rfloor}\mathbf{P}\left(\bigcap\limits_{k=1}^{\lfloor v\varphi(T)\rfloor}\{\xi(\kappa_k-)=m_k\},
\bigcap\limits_{k=1}^{\lfloor v\varphi(T)\rfloor-1}\{\check{\zeta}_{k_k}(m_{k})= h_k\}\right).
\end{aligned}
\end{equation}
It is easy to see that
\begin{equation}\label{02.03.3}
\begin{aligned}
&\sum\limits_{\substack{
\lfloor\delta \varphi(T)\rfloor<m_{1}<\infty\\
\dots\\
\lfloor\delta \varphi(T)\rfloor<m_{\lfloor v\varphi(T)\rfloor}<\infty
}}
\mathbf{P}\left(\bigcap\limits_{k=1}^{\lfloor v\varphi(T)\rfloor}\{\xi(\kappa_k-)=m_k\},
\bigcap\limits_{k=1}^{\lfloor v\varphi(T)\rfloor-1}\{\check{\zeta}_{k_k}(m_{k})= h_k\}\right)
\\
&{\qquad}\leq\sum\limits_{\substack{
\lfloor\delta \varphi(T)\rfloor<m_{1}<\infty\\
\dots\\
\lfloor\delta \varphi(T)\rfloor<m_{\lfloor v\varphi(T)\rfloor-1}<\infty
}}
\mathbf{P}\left(\bigcap\limits_{k=1}^{\lfloor v\varphi(T)\rfloor-1}\{\xi(\kappa_k-)=m_k\},
\bigcap\limits_{k=1}^{\lfloor v\varphi(T)\rfloor-1}\{\check{\zeta}_{k_k}(m_{k})= h_k\}\right)
\\
&{\qquad} =\sum\limits_{\substack{
\lfloor\delta \varphi(T)\rfloor<m_{1}<\infty\\
\dots\\
\lfloor\delta \varphi(T)\rfloor<m_{\lfloor v\varphi(T)\rfloor-1}<\infty
}}\mathbf{P}_{\lfloor v\varphi(T)\rfloor-1}(h_1,\dots,h_{\lfloor v\varphi(T)\rfloor-1}).
\end{aligned}
\end{equation}
It follows recursively from (\ref{02.03.2}) and (\ref{02.03.3}) that
\begin{equation}\label{02.03.5}
\sum\limits_{\substack{
\lfloor\delta \varphi(T)\rfloor<m_{1}<\infty\\
\dots\\
\lfloor\delta \varphi(T)\rfloor<m_{\lfloor v\varphi(T)\rfloor}<\infty
}}\mathbf{P}_{\lfloor v\varphi(T)\rfloor}(h_1,\dots,h_{\lfloor v\varphi(T)\rfloor})
\leq\left(\frac{1}{\lfloor \delta\varphi(T)\rfloor}\right)^{\lfloor v\varphi(T)\rfloor}.
\end{equation}
Recalling the definition of the set $V$ above, and using (\ref{02.03.8}), (\ref{02.03.1}), (\ref{02.03.5}) and Lemma~\ref{l5.2}, we obtain
\begin{equation} \label{new3.3}
\begin{aligned}
\mathbf{P}_{11r}  &\leq\sum\limits_{(h_1,\dots,h_{\lfloor v\varphi(T)\rfloor})\in V}\left(\frac{1}{\lfloor \delta\varphi(T)\rfloor}\right)^{\lfloor v\varphi(T)\rfloor}
=\mathbf{P}\bigg(\sum\limits_{k=1}^{\lfloor v\varphi(T)\rfloor}\zeta_k(\lfloor \delta\varphi(T)\rfloor)\leq 2a\varphi(T)\bigg)
\\
&\leq
\bigg(\frac{1}{\lfloor \delta\varphi(T)\rfloor}\bigg)^{\lfloor v\varphi(T)\rfloor}\exp(2a\varphi(T)).
\end{aligned}
\end{equation}
For any $a>0$ and $v>0$, and for sufficiently large $T$, we have
\begin{equation} \label{23.11.7}
\bigg(\frac{1}{\lfloor \delta\varphi(T)\rfloor}\bigg)^{\lfloor v\varphi(T)\rfloor}\exp(2a\varphi(T))=
\exp\bigl\{-\lfloor v\varphi(T)\rfloor\ln\lfloor \delta\varphi(T)\rfloor+2a\varphi(T)\bigr\}\leq e^{-a\varphi(T)}.
\end{equation}
By combining (\ref{A.2}), (\ref{new3.2}), (\ref{new3.3}) and (\ref{23.11.7}), we obtain the final upper bound for $\mathbf{P}_1$: for all
$a>\frac{3c\alpha\lambda}{\lambda+\mu}(e-1)$
\begin{equation} \label{23.11.12}
\mathbf{P}_1\leq \frac{T}{c\varphi(T)}\left(e^{-\frac{\alpha\mu c}{2(\lambda+\mu)}\varphi(T)}+2e^{-a\varphi(T)}\right).
\end{equation}

Let us now bound $\mathbf{P}_2$ from above. Note,
$$
\mathbf{P}_2\leq n \max\limits_{1\leq r \leq n}
\mathbf{P}\bigg(\xi_T(1)\geq x, \inf\limits_{t\in[t_{r-1},t_{r}]}\xi_T(t)\leq\delta\bigg)
=:n\max\limits_{1\leq r \leq n}
\mathbf{P}_{2r}.
$$
To find an upper bound for $\mathbf{P}_{2r}$, we first decompose the probability $\mathbf{P}_{2r}$, in terms of the number of upward jumps within the time interval $[t_{r-1},t_r]$:
$$
\begin{aligned}
\mathbf{P}_{2r}\leq\mathbf{P}\big(\xi_T(1)\geq x, \inf\limits_{t\in[t_{r-1},t_{r}]}\xi_T(t)\leq\delta,~
 &\nu_1(t_{r})-\nu_1(t_{r-1})\leq \delta \varphi(T)\big)\\
+
\mathbf{P}\big(&\nu_1(t_{r})-\nu_1(t_{r-1})> \delta \varphi(T)\big)
=:
\mathbf{P}_{2r1}+\mathbf{P}_{2r2}.
\end{aligned}
$$
It is easy to see that, if
$$
\inf\limits_{t\in[t_{r-1},t_{r}]}\xi_T(t)\leq\delta, \ \ \ \nu_1(t_{r})-\nu_1(t_{r-1})\leq \delta,
$$
then there exists $\tau\in[t_{r-1},t_{r}]$ such that $\xi_T(\tau)\leq \delta$.
Hence, the following inequality holds
\begin{equation}\label{04.03.2}
\xi_T(t_{r})\leq \xi_T(\tau)+\nu_1(t_{r})-\nu_1(\tau)\leq
\xi_T(\tau)+\nu_1(t_{r})-\nu_1(t_{r-1})\leq 2\delta.
\end{equation}
Using  (\ref{04.03.2}), we obtain
\begin{equation} \label{22.11.3}
\begin{aligned}
	\mathbf{P}_{2r1} & \leq\mathbf{P}\bigg(\xi_T(1)\geq x, \xi_T(t_{r})\leq 2\delta\bigg) =\sum\limits_{l=0}^{\lfloor 2\delta\varphi(T) \rfloor}\mathbf{P}\bigg(\xi_T(1)\geq x, \xi(T t_{r})= l\bigg) \\
	& = \sum\limits_{l=0}^{\lfloor 2\delta\varphi(T) \rfloor}\mathbf{P}\bigg(\xi(T)\geq x\varphi(T) \ \bigg|\  \xi(T t_{r}) = l\bigg)\mathbf{P}\big(\xi(T t_{r})= l\big) \\
	& \le \sum\limits_{l=0}^{\lfloor 2\delta\varphi(T) \rfloor}\mathbf{P}\bigg(\xi\left(c\varphi(T)\left(1-\frac{r}{n}\right)\right)
	\geq x\varphi(T) \ \bigg|\ \xi(0) = l\bigg) \\
	& \leq \lfloor 2\delta\varphi(T) \rfloor\mathbf{P}\bigg(\xi\left(c\varphi(T)\left(1-\frac{r}{n}\right)\right) \geq (x-2\delta)\varphi(T)\bigg).
	\phantom{\sum\limits_{l=0}^{\lfloor 2\delta\varphi(T) \rfloor}}
\end{aligned}
\end{equation}
The last inequality follows from Corollary~\ref{c1} of Lemma~\ref{l5.1}.

Applying Theorem~\ref{t1.1} to
$$
\xi_{c\left(1-\frac{r}{n}\right)\varphi(T)}(t):=\frac{\xi\left(c\left(1-\frac{r}{n}\right)\varphi(T)t\right)}{\varphi(T)}, \ t\in[0,1],
$$
it follows that, for any $2\delta<x$,
\begin{equation} \label{new3.4}
\limsup\limits_{T\rightarrow\infty}\frac{1}{\varphi(T)}\ln\mathbf{P}
\bigg(\frac{\xi\left(c\varphi(T)\left(1-\frac{r}{n}\right)\right)}{\varphi(T)}\geq x-2\delta\bigg)
=-J_k(x-2\delta),
\end{equation}
where $k:=\frac{1}{c\left(1-\frac{r}{n}\right)}$.

It is easy to verify (see formula (\ref{22.11.5})) that for any $c>\frac{x-2\delta}{\alpha}$ and for sufficiently large $n$, the maximum of the right-hand side of
inequality (\ref{new3.4}) is attained at $r=1$, and the value is equal to
\begin{equation} \label{23.11.14}
-(x-2\delta)\ln\left(\frac{\lambda+\mu}{\lambda}\right).
\end{equation}

Now, to bound $\mathbf{P}_{2r2}$, we first apply Markov's inequality
$$
\mathbf{P}_{2r2} = \mathbf{P}\left(\nu_1\left(\frac{c\varphi(T)}{n}\right)> \delta \varphi(T)\right)=
\mathbf{P}\left(e^{\frac{2a}{\delta}\nu_1\left(\frac{c\varphi(T)}{n}\right)}\geq e^{2a \varphi(T)}\right)
\leq\frac{\mathbf{E}e^{\frac{2a}{\delta}\nu_1\left(\frac{c\varphi(T)}{n}\right)}}{e^{2a \varphi(T)}}.
$$
Then, choosing
$n\geq\frac{c\alpha\lambda}{\lambda+\mu}(e^{\frac{2a}{\delta}}-1)$,
we  obtain
\begin{equation} \label{23.11.15}
\mathbf{P}_{2r2} \le \exp\left\{\frac{c\alpha\lambda}{n(\lambda+\mu)}(e^{\frac{2a}{\delta}}-1)\varphi(T)-2a\varphi(T)\right\}
\leq e^{-a\varphi(T)}.
\end{equation}
By combining (\ref{23.11.18}), (\ref{23.11.12})\,--\,(\ref{22.11.3}), (\ref{23.11.14}), and (\ref{23.11.15}), we obtain a bound for $a>\frac{3c\alpha\lambda}{\lambda+\mu}(e-1)$, $c>\frac{x-2\delta}{\alpha}$,
 $n\geq\frac{c\alpha\lambda}{\lambda+\mu}(e^{\frac{2a}{\delta}}-1)$ and $T\rightarrow\infty$:
\begin{equation} \label{23.11.22}
\begin{aligned}
\mathbf{P}(\xi_T(1)\geq x)\leq &
\frac{T}{c\varphi(T)}\left(e^{-\frac{\alpha\mu c}{2(\lambda+\mu)}\varphi(T)}+2e^{-a\varphi(T)}\right)
+ne^{-a\varphi(T)}
\\
& +\lfloor 2\delta\varphi(T) \rfloor n\exp\left\{-\varphi(T)\left((x-2\delta)\ln\left(\frac{\lambda+\mu}{\lambda}\right)+o(1)\right)\right\}.
\end{aligned}
\end{equation}
Using (\ref{23.11.22}), for sufficiently large $a$, $c$ and $n$, as $T\rightarrow\infty$ we obtain
\begin{align*}
\mathbf{P}&(\xi_T(1)\geq x)\\
&\le \left(\frac{3T}{c\varphi(T)}+n+\lfloor 2\delta\varphi(T) \rfloor n\right)
\exp\left\{-\varphi(T)\left((x-2\delta)\ln\left(\frac{\lambda+\mu}{\lambda}\right)+o(1)\right)\right\}.
\end{align*}
Thus, for any $\delta>0$ the following inequality holds
$$
\limsup\limits_{T\rightarrow\infty}\frac{1}{\varphi(T)}\ln\mathbf{P}(\xi_T(1)\geq x)\leq
-(x-2\delta)\ln\frac{\lambda+\mu}{\lambda}.
$$
By letting $\delta\rightarrow 0$, we obtain inequality (\ref{02.03.7}).

\subsubsection*{Lower bound.}
Let us bound from below $\mathbf{P}(\xi_T(1)\geq x)$, $x>0$. By definition of the process
\begin{equation} \label{new11.11.1}
\begin{aligned}
	\mathbf{P}(\xi_T(1)\geq x) &\ge \mathbf{P}\bigl( \nu_1(T)-\nu_1(T-c\varphi(T)) \geq \varphi(T)x,\nu_2(T)-\nu_2(T-c\varphi(T))=0 \bigr) \\
	&= \mathbf{P} \bigl( \nu_1(c\varphi(T)) \geq \varphi(T)x,\nu_2(c\varphi(T))=0 \bigr).
\end{aligned}
\end{equation}
Using Lemma~\ref{l5.3}, for any $c>0$ we have
\begin{align*}
\liminf\limits_{T\rightarrow\infty}
    &\frac{1}{\varphi(T)}\ln\mathbf{P} \bigr( \nu_1(c\varphi(T)) \geq\varphi(T)x,\nu_2(c\varphi(T))=0 \bigl) \\
    &=-x\ln\left(\frac{x(\lambda+\mu)}{\alpha\lambda c}\right)+x-\frac{\alpha\lambda c}{\lambda+\mu}-\frac{\alpha\mu c}{\lambda+\mu}.
\end{align*}
Choosing $c=\frac{x}{\alpha}$, we obtain
\begin{equation} \label{23.11.24}
\liminf\limits_{T\rightarrow\infty}\frac{1}{\varphi(T)}\ln\mathbf{P}(\xi_T(1)\geq x)\geq
-x\ln\frac{\lambda+\mu}{\lambda}=-I_1(x).
\end{equation}
Using (\ref{02.03.7}), (\ref{23.11.24}), and function $I_1(x)$, we obtain for $x>0$
$$
\begin{aligned}
\lim\limits_{\varepsilon \rightarrow 0}&\liminf\limits_{T \rightarrow \infty}\frac{1}{\varphi(T)} \ln \mathbf{P}(\xi_T(1)\in (x-\varepsilon,x+\varepsilon))
\\
&
=\lim\limits_{\varepsilon \rightarrow 0}\liminf\limits_{T \rightarrow \infty}\frac{1}{\varphi(T)}
\ln \big(\mathbf{P}(\xi_T(1)> x-\varepsilon)-\mathbf{P}(\xi_T(1) \geq x+\varepsilon)\big)
\\
&=\lim\limits_{\varepsilon \rightarrow 0}\liminf\limits_{T \rightarrow \infty}\frac{1}{\varphi(T)}
\ln\mathbf{P}(\xi_T(1)> x-\varepsilon) +
\lim\limits_{\varepsilon \rightarrow 0}\liminf\limits_{T \rightarrow \infty}\frac{1}{\varphi(T)}
\ln\bigg(1-\frac{\mathbf{P}(\xi_T(1) \geq x+\varepsilon)}{\mathbf{P}(\xi_T(1) > x-\varepsilon)}\bigg)
\\
&\geq \lim\limits_{\varepsilon \rightarrow 0}\liminf\limits_{T \rightarrow \infty}\frac{1}{\varphi(T)}
\ln\mathbf{P}(\xi_T(1)> x-\varepsilon)\\
&{\quad}+
\lim\limits_{\varepsilon \rightarrow 0}\liminf\limits_{T \rightarrow \infty}\frac{1}{\varphi(T)}
\ln\bigg(1-e^{-\varphi(T)(I_1(x+\varepsilon)-I_1(x-\varepsilon)+o(1))}\bigg)
\\
&=-I_1(x)+\lim\limits_{\varepsilon \rightarrow 0}\liminf\limits_{T \rightarrow \infty}\frac{1}{T}
\ln\bigg(1-e^{-\varphi(T)\left(2\varepsilon\ln\frac{\lambda+\mu}{\lambda} +o(1)\right)}\bigg)=-I_1(x).
\end{aligned}
$$
The lower bound completes the proof of the local large deviation principle.

Since $I_1(x)$ increases continuously to infinity, exponential tightness follows from inequality (\ref{02.03.7}). %$\Box$

\subsection{Proof of Theorem~\ref{t2.3}}

From Theorem~\ref{t2.1}, it follows that Theorem~\ref{t2.3} holds for
$x=0$. We now proceed to prove the theorem for $x>0$.

Let us bound from above $\mathbf{P}(\xi_T(1)\geq x)$ for $x>0$.
Applying Markov's  inequality, we obtain for any $u>0$
$$
\mathbf{P}(\xi_T(1)\geq x)\leq \mathbf{P}(\nu_1(T)\geq \varphi(T)x)
=\mathbf{P}(e^{u\nu_1(T)}\geq e^{u\varphi(T)x})
$$
$$
\leq\frac{\mathbf{E}e^{u\nu_1(T)}}{e^{u\varphi(T)x}}
=\exp\left\{-\frac{\alpha\lambda}{\lambda+\mu}T+\frac{\alpha\lambda}{\lambda+\mu}Te^{u}
-u\varphi(T)x\right\}.
$$
Choosing $u=\ln\frac{\varphi(T)}{T}$, we have
$$
\mathbf{P}(\xi_T(1)\geq x) \leq
\exp\left\{-\frac{\alpha\lambda}{\lambda+\mu}T+\frac{\alpha\lambda}{\lambda+\mu}\varphi(T)
-x\varphi(T)\ln\frac{\varphi(T)}{T}\right\}.
$$
Thus, using the condition $\lim\limits_{T\rightarrow\infty}\frac{\varphi(T)}{T}=\infty$, we obtain
$$
\limsup\limits_{T\rightarrow\infty}\frac{1}{\varphi(T)\ln\frac{\varphi(T)}{T}}\ln
\mathbf{P}(\xi_T(1)\geq x)\leq -x= -I_2(x).
$$

Let us bound from below $\mathbf{P}(\xi_T(1)\geq x)$ for $x>0$.
Using Stirling's formula, we obtain the following as $T\rightarrow\infty$
$$
\begin{aligned}
	\mathbf{P}(\xi_T(1)\geq x) &\geq \mathbf{P}(\nu_1(T)\geq \varphi(T)x)\mathbf{P}(\nu_2(T)=0)=\mathbf{P}(\nu_1(T)\geq \varphi(T)x)e^{-\frac{\alpha\mu}{\lambda+\mu}T}
\\
&\geq
\mathbf{P}\big(\nu_1(T)= \lfloor \varphi(T)x \rfloor+1\big)e^{-\frac{\alpha\mu}{\lambda+\mu}T}
=\frac{e^{-\alpha T}
\left(\frac{\alpha\lambda}{\lambda+\mu}T\right)^{\lfloor \varphi(T)x \rfloor+1}}{(\lfloor \varphi(T)x \rfloor+1)!}
\\
&\geq
\exp\left\{- \lfloor\varphi(T)x+1 \rfloor\ln(\lfloor \varphi(T)x \rfloor+1)-\alpha T
+\lfloor \varphi(T)x \rfloor\ln\frac{T\alpha\lambda}{\lambda+\mu}\right\}
\\
&=\exp\left\{-\lfloor \varphi(T)x \rfloor\ln\frac{ \lfloor\varphi(T) \rfloor}{T}(1+o(1))\right\}.
\end{aligned}
$$
Thus,
$$
\liminf\limits_{T\rightarrow\infty}\frac{1}{\varphi(T)\ln\frac{\varphi(T)}{T}}\ln
\mathbf{P}(\xi_T(1)\geq x)\geq -x= -I_2(x).
$$
The conclusion of the proof follows the same arguments as those presented at the end of the proof of Theorem~\ref{t2.2}, and therefore, we omit it here. %$\Box$

\section{Auxiliary results}

Here, we formulate and prove some auxiliary lemmas.

\begin{lemma} \label{l5.4} Let $\nu$ be random variable with Poisson distribution
	with $\mathbf{E}\nu=\beta$.
Then for any $u\in[0,1)$, $z>0$ the following inequality holds
\begin{equation} \label{5.4}
\mathbf{P}(\nu\leq z)\leq \exp\bigg\{-\beta(1-u)
 -z\ln u \bigg\}.
\end{equation}
\end{lemma}

\begin{proof}
Using Markov's inequality, for any $r>0$ we obtain
$$
\mathbf{P}(\nu\leq z)=
\mathbf{P}\big(\exp\{-r\nu\}\geq \exp\{-r z\}\big)
\leq\frac{\mathbf{E}\exp\{-r\nu\}}{\exp\{-r z\}}=
\exp\bigg\{\beta(e^{-r}-1)+r z\bigg\}.
$$
Choosing $r=-\ln u$, we obtain the inequality (\ref{5.4}).
\end{proof}

\begin{lemma}[\!\!{\cite[Lemma 4.1]{LYL}}] \label{l5.2}
 For any $a\in \mathbb{R}$ the following inequality holds true
$$
\mathbf{P}\bigg(\sum\limits_{k=1}^{\lfloor v\varphi(T)\rfloor}\zeta_k(\lfloor \delta\varphi(T)\rfloor)\leq 2a\varphi(T)\bigg)\leq
\bigg(\frac{1}{\lfloor \delta\varphi(T)\rfloor}\bigg)^{\lfloor v\varphi(T)\rfloor}\exp(2a\varphi(T)).
$$
\end{lemma}

\begin{lemma} \label{l5.3}
The family of random variables $\frac{1}{\varphi(T)}\big(\nu_1(T)-\nu_1(T-c\varphi(T))\big)$ satisfied LDP with
normalized function $\psi(T)=\varphi(T)$ and with rate function
$$
\tilde{I}(x)= \left\{
           \begin{array}{lcl}
                              \ \infty,  \text{ if }\ x\in (-\infty,0),\\
                              x\ln\left(\frac{x(\lambda+\mu)}{\alpha\lambda c}\right)-x+\frac{\alpha\lambda c}{\lambda+\mu}, \text{ if }\ x\in [0,\infty).\\
                              \end{array}
                              \right.
$$
\end{lemma}

\begin{proof}
The proof follows directly from the proof of Lemma 4.1 \cite{LYL}, so we omit it for brevity.
\end{proof}

Denote $\xi_x(t)$ the copy of random process $\xi(t)$, such that $\xi(0)=x$, $x\in \mathbb{Z}^+$.

\begin{lemma} \label{l5.1}
For any $x,y\in \mathbb{Z}^+$ there exist random processes $\tilde{\xi}(t)$ and $\hat{\xi}(t)$
such that $\tilde{\xi}(t)\stackrel{d}{=}\xi_x(t)$, $\hat{\xi}(t)\stackrel{d}{=}\xi_y(t)$, and the following equality holds
\begin{equation} \label{new5.1}
 \mathbf{P}\left(\sup\limits_{t\geq 0}|\tilde{\xi}(t)-\hat{\xi}(t)|\leq|x-y|\right)=1.
\end{equation}
\end{lemma}

\begin{proof}
Define the random variables
$$
\tilde{\zeta}_k(x,y):=\left\lfloor \frac{\zeta_k(xy)-1}{y}\right\rfloor +1, \ \ \ \hat{\zeta}_k(x,y):=\left\lfloor \frac{\zeta_k(xy)-1}{x}\right\rfloor+1,
\text{ if } \min(x,y)>0,
$$
$$
\tilde{\zeta}_k(x,y)=\zeta_k(x), \ \ \ \hat{\zeta}_k(y)=\zeta_k(y), \text{ if } \min(x,y)=0.
$$
It is straightforward to verify that the random variables $\tilde{\zeta}_k(x,y)$, for $k \in \mathbb{Z}^+$ and $xy \in \mathbb{Z}^+$, are mutually independent and do not depend on $\nu_1(t)$ or $\nu_2(t)$. Similarly, the random variables $\hat{\zeta}_k(x,y)$, for $k \in \mathbb{Z}^+$ and $xy \in \mathbb{Z}^+$, are mutually independent and also do not depend on $\nu_1(t)$ or $\nu_2(t)$.

Note that if $\min(x,y)>0$, $k\in \mathbb{N}$, $1\leq r\leq x$, $1\leq v\leq y$, $k\in \mathbb{N}$, then
$$\mathbf{P}(\tilde{\zeta}_k(x,y)=r)=\mathbf{P}(\zeta_k(xy)\in [(r-1)y+1,ry])=1/x=\mathbf{P}(\zeta_k(x)=r),$$
$$\mathbf{P}(\hat{\zeta}_k(x,y)=v)=\mathbf{P}(\zeta_k(xy)\in [(v-1)x+1,vx])=1/y=\mathbf{P}(\zeta_k(y)=v).$$
Thus, the distributions of $\tilde{\zeta}_k(x,y)$ and $\hat{\zeta}_k(x,y)$ coincide with the distribution of $\zeta_k(x)$ and $\zeta_k(y)$, respectively.

Therefore, the distributions of random processes
$$
\tilde{\xi}(t)=x+\nu_1(t)-\sum\limits_{k=0}^{\nu_2(t)}\tilde{\zeta}_k(\tilde{\xi}(\tau_k-),\hat{\xi}(\tau_k-)),
$$
$$
\hat{\xi}(t)=y+\nu_1(t)-\sum\limits_{k=0}^{\nu_2(t)}\hat{\zeta}_k(\tilde{\xi}(\tau_k-),\hat{\xi}(\tau_k-))
$$
coincide with the distributions of $\xi_x(t)$ and $\xi_y(t)$, respectively.

If $x=y$, then
\begin{equation} \label{21.11.1}
\tilde{\zeta}_k(x,y)-\hat{\zeta}_k(x,y)=0 \text{ a.s.}
\end{equation}
Let $x> y$,  $\min(x,y)>0$, then, using inequalities $\big\lfloor \frac{a}{y}\big\rfloor\leq\frac{a}{y}$,
$\big\lfloor \frac{a}{x}\big\rfloor\geq\frac{a}{x}-\frac{x-1}{x}$, $a\in \mathbb{Z}^+$
and $\zeta_k(xy)\leq xy$ a.s., we obtain
\begin{equation} \label{21.11.2}
\begin{aligned}
0 & \leq\tilde{\zeta}_k(x,y)-\hat{\zeta}_k(x,y)=\left\lfloor \frac{\zeta_k(xy)-1}{y}\right\rfloor-\left\lfloor\frac{\zeta_k(xy)-1}{x}\right\rfloor
\\
&\leq \frac{(\zeta_k(xy)-1)(x-y)}{xy}+\frac{x-1}{x}
\\
&\leq x-y - \frac{(x-y)}{xy}+\frac{x-1}{x}= x-y +1 - \frac{1}{y}  \ \ \text{ a.s.}
\end{aligned}
\end{equation}
Since $\hat{\zeta}_k(x,y)-\tilde{\zeta}_k(x,y)$ is integer, from (\ref{21.11.2}), it follows that if the inequalities $x> y$ and  $\min(x,y)>0$ hold, then
\begin{equation} \label{new5.2}
0\leq\tilde{\zeta}_k(x,y)-\hat{\zeta}_k(x,y)
\leq x-y \text{ a.s.}
\end{equation}
A similar inequality also holds for $y> x$ and $\min(x,y)>0$.

If $x> y$ and $y=0$, then
\begin{equation} \label{new5.3}
2 \leq \tilde{\zeta}_k(x,y)-\hat{\zeta}_k(x,y)=\zeta_k(x)-\zeta_k(0)
=\zeta_k(x)+1
\leq x+1  \text{ a.s.}
\end{equation}
A similar inequality also holds for $y> x$ and  $x=0$.

We complete the proof by mathematical induction.
It is obvious that for $t\in [0,\tau_1)$
$$
|\tilde{\xi}(t)-\hat{\xi}(t)|=|(x+\nu_1(t))-(y+\nu_1(t))|=|x-y|  \text{ a.s.}
$$
Let $x\geq y$. Using inequalities (\ref{21.11.1})--(\ref{new5.3}) we obtain
$$
\begin{aligned}
|\tilde{\xi}(\tau_1)-\hat{\xi}(\tau_1)|=& \Bigl|(x+\nu_1(\tau_1)-\tilde{\zeta}_1(x+\nu_1(\tau_1-),y+\nu_1(\tau_1-))
\\
& {} - (y+\nu_1(\tau_1)-\hat{\zeta}_1(x+\nu_1(\tau_1-),y+\nu_1(\tau_1-)) \Bigr|
\\
= & \Big|x-y-\Big(\tilde{\zeta}_1(x+\nu_1(\tau_1-),y+\nu_1(\tau_1-))
-\hat{\zeta}_1(x+\nu_1(\tau_1-),y+\nu_1(\tau_1-))\Big)\Big|
\\
\leq &
\begin{cases}|x-y|, \text{ if } x\geq y, \ y\neq 0,\\
\max(1,x-2), \text{ if } x> y, \ y=0,\\
0, \text{ if } x=y=0,
\end{cases}
\\
\leq & |x-y| \text{ a.s.}
\end{aligned}
$$
The case $y>x$ is treated in the same way. So the inequality
\begin{equation}\label{03.03.7}
|\tilde{\xi}(t)-\hat{\xi}(t)|\leq |x-y|  \text{ a.s.}
\end{equation}
holds for $t\in [0,\tau_1]$.

Suppose that inequality (\ref{03.03.7}) holds for $t\in [0,\tau_{n-1}]$. Then, using the inductive hypothesis, for $t\in [\tau_{n-1},\tau_n)$, we have
$$
|\tilde{\xi}(t)-\hat{\xi}(t)|=|(\tilde{\xi}(\tau_{n-1})+\nu(t)-\nu(\tau_{n-1}))-(\hat{\xi}(\tau_{n-1})+\nu(t)-\nu(\tau_{n-1}))|\leq |x-y|  \text{ a.s.}
$$
Let $\tilde{\xi}(\tau_{n-1})=z$, $\hat{\xi}(\tau_{n-1})=u$, and suppose that $z\geq u$. By the inductive hypothesis, we have $z-u\leq|x-y|$. Using inequality (\ref{21.11.1})--(\ref{new5.3}) and the inductive hypothesis, it follows that
$$
\begin{aligned}
|\tilde{\xi}(\tau_n)&-\hat{\xi}(\tau_n)| \\
& = \Bigl|\Big(\tilde{\xi}(\tau_{n-1})+\nu(\tau_n)-\nu(\tau_{n-1})-
\tilde{\zeta}_n\big(\tilde{\xi}(\tau_{n-1}) \\
&{\qquad} +\nu(\tau_n-)-\nu(\tau_{n}),\hat{\xi}(\tau_{n-1})+\nu(\tau_n-)-\nu(\tau_{n-1})\big)\Big) \\
&\phantom{=\big|}-\Big(\hat{\xi}(\tau_{n-1})+\nu(\tau_n)-\nu(\tau_{n-1})
-\hat{\zeta}_n\big(\hat{\xi}(\tau_{n-1}) \\
&{\qquad\quad} +\nu(\tau_n-)-\nu(\tau_{n}),\hat{\xi}(\tau_{n-1})+\nu(\tau_n-)-\nu(\tau_{n})\big)\Big)\Bigr| \\
&=\Big(z-u-\Big(\tilde{\zeta}_n\big(z+\nu(\tau_n-)-\nu(\tau_{n}),u+\nu(\tau_n-)-\nu(\tau_{n-1})\big) \\
&\phantom{=\big(} -\hat{\zeta}_n\big(z+\nu(\tau_n-)-\nu(\tau_{n}),u+\nu(\tau_n-)-\nu(\tau_{n})\big)\Big)\Big) \\
&\leq |z-u|\leq |x-y| \ \  \text{ a.s.}
\end{aligned}
$$
The case $u>z$ is treated in the same way. Thus, the inequality (\ref{03.03.7}) holds for $t\in [0,\tau_{n}]$ and $n\in \mathbb{N}$.

Since $\lim\limits_{n\rightarrow\infty}\tau_n=\infty$ a.s., the proof is complete.
\end{proof}

\begin{corollary} \label{c1} It follows directly from Lemma~\ref{l5.1} that for any integers $0\leq x<y$, $z\in \mathbb{R}$, and $t\geq 0$, the following inequality holds:
$$
\mathbf{P}(\xi_y(t)>z)\leq \mathbf{P}(\xi_x(t)>z-|x-y|).
$$
\end{corollary}

\begin{proof}
Lemma \ref{l5.1} implies
\begin{align*}
\mathbf{P}(\xi_y(t)>z)&=\mathbf{P}(\tilde{\xi}_y(t)>z)=
\mathbf{P}\left(\hat{\xi}_y(t)>z,\sup\limits_{t\geq 0}|\tilde{\xi}_x(t)-\hat{\xi}_y(t)|\leq|x-y|\right)
\\ &\leq\mathbf{P}\left(\hat{\xi}_y(t)>z,\tilde{\xi}_x(t)>z-|x-y|\right)
\leq\mathbf{P}\left(\tilde{\xi}_x(t)>z-|x-y|\right).
\qedhere
\end{align*}
\end{proof}

\subsection*{Acknowledgments} This research was funded by RSF grant number 24-28-01047. Yambartsev A. thanks the S\~ao Paulo Research Foundation FAPESP, Brazil, for support under grant 2023/13453-5.

\subsection*{Author Contributions} The work on the introduction, literature review, and editing was carried out by A.Y. and O.L.; A.L. and O.L. did the proofs of the main theorems and auxiliary results.

%%% REFERENCES %%%
{\small\bibliography{cimart}}
% Please, do not change the above line and do not insert your references
% into this file.  Instead, insert your references into the cimart.bib file.
% See cimart.bib for further instructions.

\EditInfo{December 5, 2024}{March 10, 2025}{Ana Cristina Moreira Freitas}

\end{document}